\documentclass[12pt,oneside]{article}
\usepackage{graphicx} 
\usepackage{epstopdf}
\usepackage{verbatim}
\usepackage{amsmath} 
\usepackage{amsthm}
\usepackage{amssymb}
\usepackage{wasysym}
\usepackage[overload]{textcase}
\usepackage{vmargin}
\usepackage{dsfont}
\usepackage{tikz}
\usetikzlibrary{arrows,shapes,backgrounds,positioning,shadows}
\usepackage[latin1]{inputenc}
\usepackage{epigraph}
\usepackage{multirow}

\theoremstyle{plain}
\newtheorem{theorem}{Theorem}
\newtheorem{proposition}{Proposition}
\newtheorem{corollary}{Corollary}
\theoremstyle{definition} 
\newtheorem{definition}{Definition}

\newcommand{\set}[1]{\lbrace #1 \rbrace} 
\newcommand{\seq}[1]{\left( #1 \right)} 

\newcommand{\sse}{\subseteq}

\newcommand{\reals}{\mathbb{R}}

\newcommand{\ints}{\mathbb{Z}}
\newcommand{\nats}{\mathbb{N}}

\newcommand{\pr}[1]{\text{Prob} \left( #1 \right)} 

\newcommand{\ith}[1]{#1^{\text{th}}} 

\newcommand{\mSet}{S}
\newcommand{\mGrp}{\mathcal{G}}
\newcommand{\mEq}{\sim}
\newcommand{\mMod}{\mSet \, /  \, \mGrp}
\newcommand{\opo}{\preceq}
\newcommand{\ipo}{\opo_{\mGrp}}
\newcommand{\wpo}{\opo_{\mGrp, w}}
\newcommand{\spo}{\opo_{\mGrp, s}}
\newcommand{\sym}{\mathcal{S}}

\newcommand{\ssc}{\text{SC}}
\newcommand{\subMajor}{\prec_w}
\newcommand{\dtv}{d_{\textsc{tv}}}
\newcommand{\meet}{\wedge}

\title{Two Musical Orderings}

\author{Marcus Pendergrass\footnote{Department of Mathematics and Computer Science, Hampden-Sydney College, Hampden-Sydney, Virginia, USA, email: mpendergrass@hsc.edu}}
\date{\today}

\begin{document}
\maketitle

\epigraph{{The science of pure mathematics, in its modern developments, may claim to be the most original creation of the human spirit. Another claimant for this position is music.}}{\textit{Mathematics as an Element in the History of Thought}\\\textsc{Alfred North Whitehead}}
\begin{abstract}
We make some general observations about partial orders on quotient spaces, and explore their use in music theory, in two different contexts.
In the first, we show that many of the most familiar chord and scale types in Western music appear as extremal elements in the partial order induced by set inclusion on pitch class sets of $T_n$-type.  
In the second, we propose a partial order that models the brightness aspect of timbre.  We use this order to compare the brightness of six wind instruments, and find that the results conform to intuition.  We also use the order to pose sound design problems of a certain type, which can be solved efficiently using linear programming.

\vspace*{0.25cm}

\vspace*{0.25cm}

\bigskip
\end{abstract}

\section{Introduction} \label{s:intro}
Orderings of various kinds are explicit or implicit in many of the concepts in music theory.  
Perhaps the most prominent examples of this are the recent mathematical theories of voice leading.  
In \cite{Callender2008} Callender and Tymoczko lay down general principles that orderings should satisfy to qualify as reasonable measures of voice leading size (also see \cite{Tymoczko2006, Tymoczko2008}), while in \cite{Hall} Hall and Tymoczko focus on the familiar partial order of submajorization as way of comparing voice leadings.
Indeed, orderings of one sort or another arise naturally whenever there is a notion of size or precedence among objects in a musical space.  
Paying careful attention to the order properties of the space can lead to new musical insights, and can bring powerful mathematical techniques to bear on musical problems.

In this paper we study two partial orders that illustrate this point.  
The first, presented in section \ref{s:pcOrder} below, is the partial order induced by set inclusion on the set classes of $T_n$-type from post-tonal theory.  
While this ordering is known to music theorists \cite{Straus}, we believe its {explanatory} power has not been fully appreciated.  
In particular, we will show that many of the most important scale and chord types in Western music appear as minimal elements in certain natural suborders of this partial order.  

The second partial order is presented in section \ref{s:timbreOrder}, where we define a class of orderings on musical timbres.  Timbre is a notoriously difficult notion to quantify, involving transient effects (e.g. attack, release), steady-state effects, as well as more complex psychoacoustic and even cultural effects \cite{Alluri, Fales, Marozeau2003, Slawson}.  Here, our only aim is to model some common and relatively simple musical judgements about timbre.  One judgement that musicians often make is that a certain instrument is ``brighter'' than another; for example, a trumpet is often thought to be brighter in tone than a French horn.  We will unpack the meaning of such a judgement using a partial order on an appropriately defined musical space.  We will show that our approach generalizes to other aspects of timbre.  Finally, we will show how our approach can be used solve certain sound design problems, such as ``of all instruments no brighter than a trumpet, which has a timbre that is closest to that of an oboe?''

We preface the discussion to follow with some general observations about partial orders on quotient spaces, as they are relevant both to the ordering on set classes, to submajorization, and to other applications as well.

\section{Partial Orders on Quotient Spaces} \label{s:genFrame}
Our basic setting is a set $\mSet$ of musical objects, along with a {partial order} $\opo$ on $\mSet$.  The partial order models some notion of size or precedence among the objects in $\mSet$.  
In addition, we posit a group $\mGrp$ of transformations mapping $\mSet$ into itself.  
An equivalence relation $\mEq$ on $\mSet$ is defined by $a \mEq b$ if and only if there exists a transform $T \in \mGrp$ with $Ta = b$.  
We denote the equivalence class of $a \in \mSet$ by $A = [a]$, and the set of distinct equivalence classes in $\mSet$ is denoted by $\mMod$. 
When does the order on $\mSet$ give rise to an order on $\mMod$?  
\begin{definition}
The \emph{strong induced relation} $\spo$ on $\mMod$ is defined by
\begin{equation} \label{e:spo}
\text{$A \spo B$ if and only if for all $a \in A$ there exists $b \in B$ such that $a \opo b$.}
\end{equation}
The \emph{weak induced relation} $\wpo$ on $\mMod$ is defined by
\begin{equation} \label{e:wpo}
\text{$A \wpo B$ if and only if there exists $a \in A$ and $b \in B$ such that $a \opo b$.}
\end{equation}
\end{definition}

It is clear that if $A$ precedes $B$ in the strong relation, then it does so in the weak relation as well.  Also, since the identity map is in $\mGrp$, both the weak and strong relations are reflexive.  But, in general, neither of these relations is an actual ordering of $\mMod$.  Under certain conditions, however, they are.

\begin{definition}
The semigroup $\mGrp$ is said to be \emph{increasing} on the partial order $\seq{\mSet, \opo}$ if for all $T$ in $\mGrp$, and all $a, b \in \mSet$, whenever $a \opo b$ in $\mSet$, then $Ta \opo Tb$ as well.  The semigroup $\mGrp$ is said to \emph{act transversely} on $\mSet$ if for all $T \in \mGrp$, and all $a \in \mSet$, whenever $Ta \opo a$, then in fact $Ta = a$.  
\end{definition}

Note that if $\mGrp$ is a group that acts transversely on $\mSet$, then $a$ and $Ta$ are always either incomparable, or identical.  It follows in this case that all equivalence classes in $\mMod$ are antichains in the partial order on $\mSet$.  
\begin{theorem} \label{t:inducedOrder}
Let $\mGrp$ be a group acting on the partial order $\seq{\mSet, \opo}$.  Then:
\begin{enumerate}
\item The strong relation is a preorder on $\mMod$.  
\item If $\mGrp$ is increasing on $\mSet$, then the strong and weak relations are identical.
\item If $\mGrp$ acts transversely on $\mSet$, then the strong relation is a partial order on $\mMod$.  
\end{enumerate}
\end{theorem}
\begin{proof}
\ 
\begin{enumerate}
\item It is obvious that the strong relation is transitive.  We have already observed that it is reflexive, hence it is a preorder.

\item We have already observed that $A \spo B$ implies $A \wpo B$.  For the reverse implication, let $a$ be an arbitrary element of $A$, and assume $A \wpo B$.   Then there exist $a_0 \in A$ and $b_0 \in B$ with $a_0 \opo b_0$.  Since $a$ and $a_0$ are in the same equivalence class, there exists $T \in \mGrp$ with $Ta_0 = a$.  Since $T$ is increasing on $\mSet$, we have $a \opo Tb_0$.  Since $Tb_0 \in B$, we conclude that $A \spo B$.

\item To show the strong relation is a partial order, we need only verify that the antisymmetric property holds.  Suppose that $A \spo B$ and $B \spo A$.  Then there exists $a, a^{\prime} \in A$ and $b \in B$ with $a \opo b \opo a^{\prime}$. Since $a$ and $a^{\prime}$ are in the same equivalence class, there exists $T \in \mGrp$ with $Ta^{\prime} = a$.  Thus $Ta^{\prime} \opo a^{\prime}$, and since $\mGrp$ acts transversely, we have $a = a^{\prime}$.  Hence $b = a$, and thus $A = B$. 

\end{enumerate}

\end{proof}
Note that if $\mGrp$ is both increasing and acts transversely on $\mSet$, then by Theorem \ref{t:inducedOrder} the weak and strong relations are identical, and form a partial order on $\mMod$.  
In such cases we will simply write ``$\ipo$'' for the induced partial order.  

The following special case underlies the discussion of the subset/superset order on set classes in section \ref{s:pcOrder}.
First, suppose the group $\mGrp$ acts on a set $\mSet_0$, and that $\mSet$ is some collection of subsets of $\mSet_0$.   Extend the action of $\mGrp$ to $\mSet$ naturally, i.e.
\begin{equation} \label{e:extension}
Ta = \set{Tx: x \in a}, \quad a \in \mSet, \ T \in \mGrp
\end{equation}
Let the partial order on $\mSet$ be given by set inclusion.  In this setting, when is the induced relation on $\mMod$ a partial order?  If $\mSet$ consists of \emph{finite} subsets of $\mSet_0$, the answer is ``always.''
\begin{corollary} \label{c:finiteSetCor}
Let $\mSet$ be the collection of all {finite} subsets of some nonempty set $\mSet_0$, and let the partial order on $\mSet$ be given by set inclusion.  Let $\mGrp$ be any group acting on $\mSet_0$, and extend the action of $\mGrp$ to $\mSet$ naturally via \eqref{e:extension}.  Then the induced relations \eqref{e:spo} and \eqref{e:wpo} are identical, and form a partial order on $\mMod$.
\end{corollary}
\begin{proof}
$a \subseteq b$ implies that $Ta \subseteq Tb$ for all $T \in \mGrp$, so $\mGrp$ is increasing on $S$.  If $Ta \subseteq a$ but $Ta \ne a$ and $a$ is finite, then by the pigeonhole principle there would have to exist distinct $s_1$ and $s_2$, both in $a$, with $Ts_1 = Ts_2$.  But this would contradict the invertibility of the transform $T \in \mGrp$.  Hence, then inclusion cannot be strict, and thus $\mGrp$ acts transversely on $\mSet$.  Now apply Theorem \ref{t:inducedOrder}.
\end{proof}

In the general case when $\mSet$ may contain infinite subsets of $\mSet_0$, more restrictive conditions are required.  The next corollary delineates one such case.  Recall that the group $\mGrp$ acting on $\mSet_0$ is said to be \emph{simply transitive} if, for every pair $\seq{x,y} \in \mSet_0 \times \mSet_0$, there exists a unique $T \in \mGrp$ such that $Tx = y$.  If $\mGrp$ is simply transitive on $\mSet_0$, then the pair $\seq{\mSet_0, \mGrp}$ is a \emph{generalized interval system} as defined by Lewin \cite{Lewin}.  Say that a subset $a$ of $\mSet_0$ has the \emph{fixed point property} with respect to $\mGrp$ if, for all $T \in \mGrp$, $Ta \sse a$ implies that there exists $x \in a$ such that $Tx = x$.
\begin{corollary} \label{c:fixedPointCor}
Let $\mSet_0$ be a nonempty set, and let $\mGrp$ be a group whose action on $\mSet_0$ is simply transitive.  Let $\mSet$ be a collection of subsets of $\mSet_0$, ordered by set inclusion, and extend the action of $\mGrp$ to $\mSet$ naturally.  If each of the sets in $\mSet$ has the fixed point property with respect to $\mGrp$, then the induced relations \eqref{e:spo} and \eqref{e:wpo} are identical, and form a partial order on $\mMod$.
\end{corollary}
\begin{proof}
The proof that $\mGrp$ is increasing on $\mSet$ is as before.  To show that $\mGrp$ acts transversely on $\mSet$, let $T \in \mGrp$, and suppose that $Ta \sse a$ for some $a \in \mSet$.  Since $a$ has the fixed point property, there exists $x \in a$ with $Tx = x$.  Simple transitivity now implies that $T$ must be the identity mapping.  Hence $Ta = a$.
\end{proof}

As an example application of Corollary \ref{c:fixedPointCor}, let $\mSet_0$ be any Euclidean space, let $\mSet$ be any collection of compact, convex subsets of $\mSet_0$, and assume that the functions in $\mGrp$ are all continuous on $\mSet_0$.  Then by Brouwer's fixed point theorem each set in $\mSet$ has the fixed point property with respect to $\mGrp$.  Hence, if $\mGrp$ is simply transitive on $\mSet_0$, the induced relation on $\mMod$ is a partial order.

Although we do not pursue it in depth here, there is an application of Theorem \ref{t:inducedOrder} to product orders that is closely tied to submajorization.  Let $\mSet = \mSet_{0}^{n}$, where $\mSet_0$ is endowed with a partial order $\opo$.  The order on $\mSet_0$ extends to a partial order on $\mSet$ component-wise.  Let $\sym_n$ denote the symmetric group of all permutations on $n$ elements.  Any subgroup $\mGrp$ of $\sym_n$ acts naturally on $\mSet$ via
\begin{equation*}
\sigma a = \seq{a_{\sigma(i)}, 0 \le i \le n-1}, \quad a \in \mSet, \  \sigma \in \mGrp.
\end{equation*}
\begin{corollary} \label{c:componentWiseCor}
Let $\opo$ be a partial order on a nonempty set $\mSet_0$, let $\mSet = \mSet_{0}^{n}$ have the corresponding component-wise partial order, and let $\mGrp$ be a subgroup of the symmetric group $\sym_n$, acting naturally on $\mSet$.  Then the induced relations \eqref{e:spo} and \eqref{e:wpo} are identical, and form a partial order on $\mMod$.
\end{corollary}
\begin{proof}
Any permutation acting on $\mSet$ is increasing with respect to the component-wise order.  To see that $\mGrp$ acts transversely, let $\sigma$ be any permutation of $n$ elements, and $a \in \mSet$, and assume $\sigma a \opo a$.  We proceed by induction on $n$.  If $n = 1$, then clearly $\sigma a = a$.  Now assume inductively that, for some $n \in \nats$, $\sigma a \opo a$ implies $\sigma a = a$ whenever then length of $a$ is no more than $n$.  Let $a$ have length $n+1$, and assume that $\sigma a \opo a$.  Let $m$ be the minimum value among the components of $a$, and let $I$ be the set of indices $i$ with $a_i = m$.  Then because $\sigma a \opo a$ in the component-wise order, $\sigma$ maps $I$ into itself.  Therefore $\sigma a$ and $a$ agree on the positions in $I$.  Now let $a^{\prime}$ be the restriction of $a$ to the indices not in $I$, and let $\sigma^{\prime}$ be the corresponding restriction of $\sigma$.  Then $\sigma^{\prime}$ is a permutation since $\sigma$ maps $I$ into itself, and we have $\sigma^{\prime} a^{\prime} \opo a^{\prime}$.  So by the inductive assumption $\sigma^{\prime} a^{\prime} = a^{\prime}$.  Thus $\sigma a$ and $a$ agree in all the positions not in $I$ as well.  Hence $\sigma a = a$, completing the induction.
\end{proof}
To see the relationship with submajorization, take $S_0 = \reals$ to be pitch space, with the natural ordering, and let $\mGrp = \sym_n$ be the full symmetric group on $n$ elements.  Then $\mMod = \reals^n \, / \, \sym_n$, which is the orbifold of all multisets of $n$ pitches, as in \cite{Hall}.  By Corollary \ref{c:componentWiseCor} component-wise order on $\reals^n$ induces a partial order $\ipo$ on $\reals^n \, / \, \sym$.  Submajorization is the extension $\subMajor$ of the partial order $\ipo$ defined by
\begin{equation*}
A \subMajor B \text{ if and only if } f(A) \ipo f(B)
\end{equation*}
where $f$ is the mapping from $\reals^n \, / \, \sym_n$ into itself defined by
\begin{equation*}
f(A) = \seq{a_{[1]},a_{[1]}+a_{[2]},a_{[1]}+a_{[2]}+a_{[3]}, \dots, a_{[1]}+a_{[2]}+ \dots + a_{[n]}}
\end{equation*}
and $a_{[i]}$ denotes the $i^{\text{th}}$-largest element of the multiset $A$.  The function $f$ is one-to-one, but not onto, from which it follows that submajorization is a proper extension of $\ipo$.

\section{Subset/Superset Order on Pitch Set Classes} \label{s:pcOrder}
We now apply the framework just developed to the pitch class sets of post-tonal theory.
The pitch classes under octave equivalence in the $12$-tone system are identified with $\ints_{12}$, called \emph{pitch class space}.  Subsets of $\ints_{12}$ correspond to \emph{pitch class sets}.  There are twelve distinct transpositions of pitch class space, namely $T_n x = x + n$ for $n, x \in \ints_{12}$, where the addition is modulo $12$.  These transpositions form a group $\mGrp$ that is itself isomorphic to $\ints_{12}$.  In line with Corollary \ref{c:finiteSetCor}, take $\mSet = 2^{\ints_{12}}$, the set of all subsets of pitch class space, and extend the action of $\mGrp$ to $\mSet$ naturally, as before.  The quotient space $\mMod = 2^{\ints_{12}} \, / \, \ints_{12}$ identifies pitch class sets that are transpositionally related.  
(So for instance all the diatonic collections are represented by a single equivalence class, all the octotonic collections are represented by another class, and so on.)
By Corollary \ref{c:finiteSetCor} the partial order of set inclusion on $2^{\ints_{12}}$ induces a partial order on $2^{\ints_{12}} \, / \, \ints_{12}$.
The elements of $2^{\ints_{12}} \, / \, \ints_{12}$ are precisely the \emph{set classes of $T_n$-type} from post-tonal theory, and the induced order is called the \emph{subset/superset ordering}. \cite[pp. 53, 96]{Straus}.  

A set class $A \in \ 2^{\ints_{12}} \, / \, \ints_{12}$ will be represented as $A = \set{0, a_1, a_2, \dots, a_{n-1}}$, where the pitch classes $a_i$ are listed in increasing order.
The interval from $a_k$ to $a_{k+1}$ will be called a \emph{scalar second}, the interval from $a_k$ to $a_{k+2}$ is a \emph{scalar third}, and so forth.  (Index arithmetic is modulo $n$ here.)

Denote the suborder of set classes whose scalar seconds span no more than $k$ semitones by $\ssc_k$.
First consider $\ssc_2$ as a suborder of $\ 2^{\ints_{12}}  \, / \, \ints_{12}$.  The elements of $\ssc_2$ are set classes whose step sizes are either $1$ or $2$ semitones.  These are precisely the classes that satisfy Tymoczko's ``diatonic seconds'' constraint in \cite{Tymoczko2004}, and, as he points out, the minimal elements are the classes in $\ssc_2$ that contain no consecutive semitones.  These are the whole tone collection, the diatonic collection, the octatonic collection, and the melodic minor (or acoustic) collection.
As $k$ increases, the minimal elements of $\ssc_k$ become more ``chord-like.''  This is explained by the following theorem, which characterizes the minimal set classes in the suborders $\ssc_k$ for $N$-tone equal temperament.

\begin{proposition} \label{t:minimalEltsSetClassOrder}
Consider the partial order induced by set inclusion on $2^{\ints_N} \, / \, \ints_N$, and let $\ssc_k$ be the suborder consisting of the set classes whose scalar seconds span no more than $k$ semitones.  Then a class $A \in \ssc_k$ is minimal in the suborder if and only if every scalar third in $A$ spans at least $k+1$ semitones.
\end{proposition}
\begin{proof}
If $A \in \ssc_k$ has a scalar third spanning $k$ or fewer semitones, then eliminate the middle pitch class in that third to produce a new set class $B$.  This set class is still in $\ssc_k$, and $B \sse A$ in the induced order.  Hence $A$ is not minimal.  On the other hand, if every scalar third in $A$ spans at least $k+1$ semitones, then eliminating the middle pitch class in any third results in a class $B$ that is not in $\ssc_k$.  Hence $A$ is minimal in the suborder on $\ssc_k$.  
\end{proof}

Table 1 lists the minimal elements in each $\ssc_k$, $2 \le k \le 5$, for $12$-tone equal temperament.  Almost every set class in this table has played a prominent role in either tonal, non-tonal, or jazz music.

\begin{table}
\begin{center}
\caption{Minimal Elements of $\ssc_k$ for $12$-Tone Equal Temperament. }
{
\begin{tabular}{|c|l|l|}
\hline
Suborder & Minimal Set Classes & Comment \\ \hline
\multirow{4}{*}{$\ssc_2$} 
 & $\set{0,1,3,4,6,7,9,10}$ & octatonic scale\\
 & $\set{0,2,3,5,7,9,11}$ & melodic minor scale\\
 & $\set{0,2,4,5,7,9,11}$ & diatonic scale\\
 & $\set{0,2,4,6,8,10}$ & whole tone scale\\ \hline
\multirow{7}{*}{$\ssc_3$} 
 & $\set{0,1,4,5,8,9}$ & symmetric scale\\
 & $\set{0,2,4,6,8,10}$ & whole tone scale\\ 
 & $\set{0,2,4,7,9,11}$ & pentatonic scale\\
 & $\set{0,2,4,7,10}$ & dominant ninth chord\\ 
 & $\set{0,3,4,7,9}$ & a blues scale\\ 
 & $\set{0,3,4,7,10}$ & e.g. C$7\sharp 9$\\ 
 & $\set{0,3,6,9}$ & diminished chord\\ \hline 
\multirow{7}{*}{$\ssc_4$} 
 & $\set{0,3,6,9}$ & diminished chord\\ 
 & $\set{0,3,6,10}$ & half-diminished chord\\
 & $\set{0,3,7,10}$ & minor seventh chord\\
 & $\set{0,4,6,10}$ & e.g. C$7\flat 5$\\ 
 & $\set{0,4,7,10}$ & dominant seventh chord\\
 & $\set{0,4,7,11}$ & major seventh chord\\
 & $\set{0,4,8}$ & augmented triad \\ \hline
\multirow{7}{*}{$\ssc_5$} 
 & $\set{0,3,6,9}$ & diminished chord\\ 
 & $\set{0,3,7}$ & minor triad \\
 & $\set{0,4,6,10}$ & e.g. C$7\flat 5$\\ 
 & $\set{0,4,7}$ & major triad \\
 & $\set{0,4,8}$ & augmented triad \\
 & $\set{0,5,6,11}$ & symmetric chord\\
 & $\set{0,5,10}$ & quartal triad \\ \hline
\end{tabular}
}
\end{center}
\label{tb:maxSetClasses}
\end{table}

Clearly this approach has further generalizations.
The set of all possible pitch classes under octave equivalence is modeled by the circle $S^1$.  In line with Corollary \ref{c:finiteSetCor}, take $\mSet_0 = S^1$ and $\mSet$ to be the collection of all finite subsets of $\mSet_0$, ordered by set inclusion.  
Elements of $\mSet$ are \emph{generalized pitch class sets}.  
If $\mGrp$ is any group of transformations acting on $\mSet_0$, then $\mGrp$ extends naturally to a group acting on $\mSet$, as in Corollary \ref{c:finiteSetCor}, and set inclusion induces a partial order on $\mMod$.  The elements of $\mMod$ are \emph{generalized set classes}, and set class $A$ precedes set class $B$ in the induced order if and only if there exists pitch class sets $a \in A$ and $b \in B$ with $a \subseteq b$.
If $\mGrp$ is the group of all transpositions of pitch class space (geometrically, rotations of the circle), then this partial order contains $2^{\ints_N} \, / \, \ints_N$ as a suborder, for all $N \in \nats$.  Of course one may consider other groups $\mGrp$, and thereby obtain other partial orders.

\section{A Timbral Partial Order} \label{s:timbreOrder}
\emph{Steady state timbre} refers to those aspects of timbre apart from transient effects such as attack or release.  For instance, if one focuses on the sound of a sustained note from a trombone, one is attending to the steady-state timbre of the instrument.  Timbral qualities like ``brightness,'' ``warmth,'' and so on, often refer to the steady-state aspects of timbre.  
Following Lewin \cite[pp. 82-85]{Lewin} and many others, we adopt a simple \emph{discrete power spectrum} model for steady-state timbre.  Essentially, we assume that the steady-state timbre of a musical instrument is characterized by the amount of power in the fundamental frequency being played, and in each of its harmonics.  Of course in reality this is far from the complete picture.  
For instance, the power spectrum may change depending on the register in which the instrument is being played, and there may also be non-harmonic components in the power spectrum.  Nonetheless, the discrete power spectrum remains a useful first-order approximation to steady-state timbre.

Given a steady-state timbre, let $a_k \ge 0$ denote the proportion of total signal power in the $\ith{k}$ harmonic of the fundamental, $k = 1, 2, \cdots$.  
(Equivalently, since timbre is independent of loudness, we normalize the total power to 1.)
Assuming a maximum number $n$ of harmonics, the set of all steady-state timbres is identified with
\begin{equation*}
\mSet = \set{a \in \reals^n : a_k \ge 0, \, \sum a_k = 1},
\end{equation*}
the set of all probability vectors of length $n$.  We will refer to vectors $a \in \mSet$ as \emph{timbral vectors}.

The connection between timbre and discrete probability is a running theme in our model.  It has a musical interpretation in terms of \emph{granular synthesis} \cite[chapter 3]{Roads2002}.  Let $g_k$ be a short pulse or \emph{grain} of sound, whose frequency is the $\ith{k}$ harmonic of the fundamental.  Given a timbral vector $a$, build a signal $s$ by selecting grain $g_k$ with probability $a_k$, and repeating, delaying successive grains by a small amount.  Then the power spectrum of the signal $s$ will be approximated by the timbral vector $a$ (modulo envelope effects).  

An aspect of timbre that is often mentioned in connection with orchestral instruments is ``brightness.''  Brightness is associated with the presence of significant power in the higher harmonics of the power spectrum.  
Acousticians tend to associate brightness with some measure of center of the power spectrum; see \cite{Marozeau2003, Marozeau2006}, for instance.  
However, if brightness is measured by any scalar quantity, there is the immediate consequence that \emph{any} two timbres are comparable in terms of brightness.  
But this seems at odds with intuition.  
Must it necessarily be the case that either a clarinet is brighter than a dulcimer or \textit{vice versa}, for instance?  
We would prefer a model which leaves open the possibility the certain timbres are simply incomparable in terms of brightness.

We now define a partial order on timbres that encodes an idea of brightness.
Say that a timbre $b \in \mSet$ is brighter than another timbre $a$ if for each $k = 1, 2, \cdots, n$ we have
\begin{equation} \label{e:brightPO}
\sum_{i = k}^{n} a_i \le \sum_{i = k}^n b_i.
\end{equation}
If $b$ is brighter than $a$, we write $a \opo b$.
In this order, $b$ is brighter than $a$ precisely when $b$ has more power in the highest $j$ harmonics than does $a$, {for each $j = 1, 2, \cdots , n$}.  
The probabilistic interpretation of the brighter-than order is this: let $K_a \in \set{1 , 2, \dots, n}$ be a random integer (harmonic) chosen according to the distribution (timbre) $a \in \mSet$, and let $K_b$ be a similar random sample from $b$.  Then $b$ is brighter than $a$ precisely when
$$\pr{K_a > \alpha} \le \pr{K_b > \alpha}$$  
for all $\alpha$.  In other words, $K_a$ precedes $K_b$ in the partial order of {stochastic dominance}: one is always more likely to sample a high harmonic from $K_b$'s distribution than from $K_a$'s.

Figure \ref{fig:spectra20-po} shows the relationships between six instruments in the brighter-than order.  
The spectra for these instruments were extracted from recordings made by Lawrence Fritts at the University of Iowa Electronic Studios \cite{Fritts}.
A directed arrow from one instrument to another means that the second instrument is brighter than the first.  
The flute, oboe, and trumpet are the maximal elements in the suborder, and the alto saxophone, clarinet, and horn are minimal elements.  
Note that the flute and oboe both dominate all the minimal elements, but the trumpet only dominates the horn.  This is due to the rapid decay of the highest harmonics in the trumpet spectrum.
\begin{figure}
\begin{center}
\includegraphics[scale=.47]{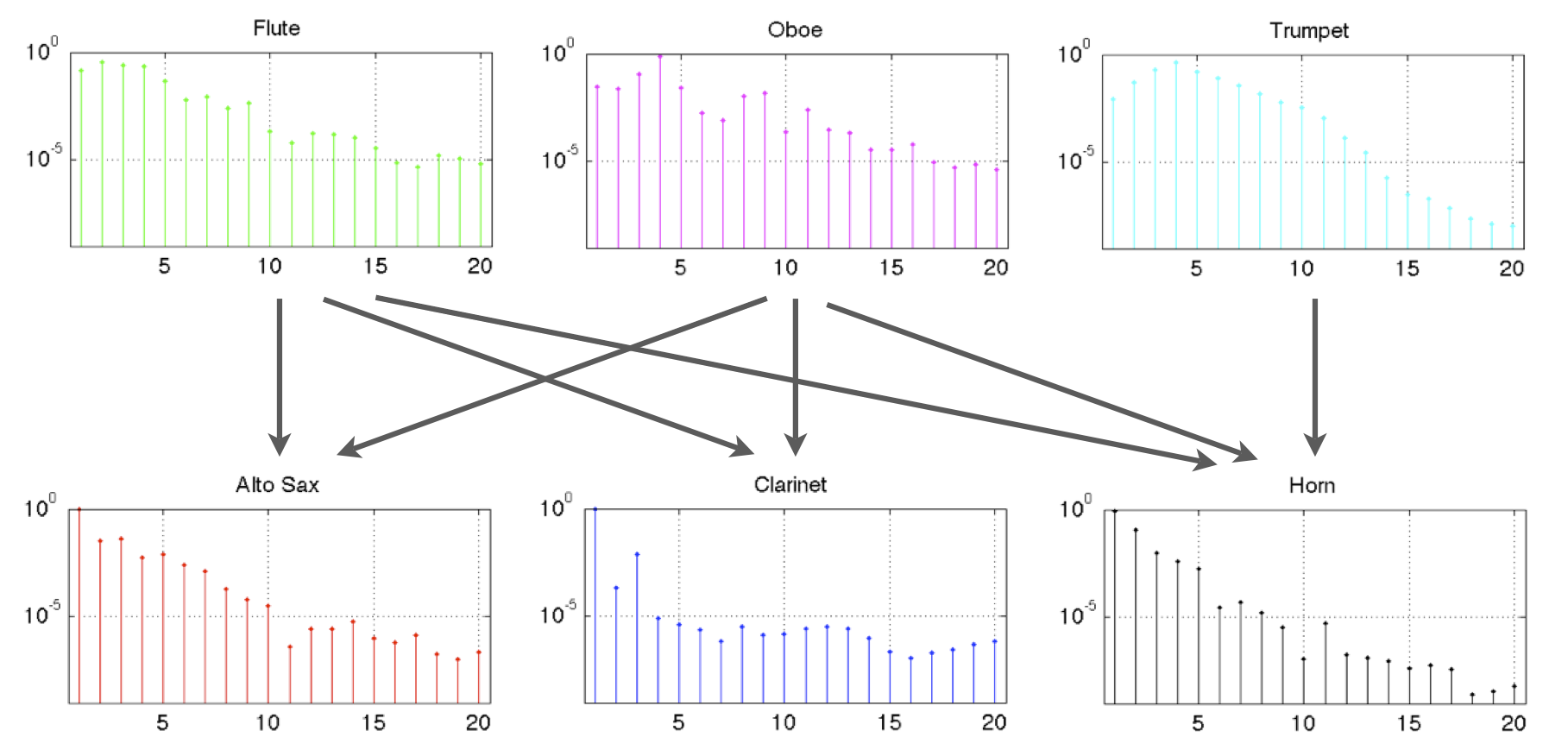} 
\end{center}
\caption{Six instruments in the the brighter-than order.}
\label{fig:spectra20-po}
\end{figure}

Obviously one can question how well the brightness partial order defined by \eqref{e:brightPO} models actual listeners' judgements about timbral brightness.  
Such judgements depend on a variety of factors, and not all listeners will come to the same conclusion.
It is certainly possible to tweak the definition \eqref{e:brightPO} in various ways to better approximate the judgements most listeners would make on a set of pre-defined test cases.
Indeed, this would be an interesting investigation to undertake.
The larger point, however, is that by modeling brightness as an ordering, rather than as a scalar quantity, we have more ``degrees of freedom'' at our disposal to make the model realistic.

Brightness is not the only timbral quality that can be expressed as a partial order.  
In general, if $H$ is any $n$-by-$n$ nonnegative, nonsingular matrix, then a partial order on timbral vectors is defined by $a \opo b$ if and only if $Ha \le Hb$ component-wise.  
The brighter-than partial order is of this type, with
\begin{equation} \label{e:brightH}
H_{ij} = 
\begin{cases}
1 \quad &\text{if $n-i+1 \le j \le n$} \\
0 \quad &\text{otherwise}
\end{cases}
\end{equation}

As an application of these ideas, consider a sound design problem of the following type:
\textit{among all instruments that are no brighter than a trumpet, which has the timbre that is closest to an oboe?}
The answer (or answers) to this problem depend on how one defines ``closest to,'' that is, on the choice of a metric for the space $\mSet$ of all timbral vectors.  
One musically meaningful measure of distance between timbres is defined by taking the maximum power discrepancy between the two timbres over all possible subsets of harmonics.
In probabilistic terms, this is the \emph{total variational distance} between probability distributions \cite[p. 128]{Grimmett}, denoted by $\dtv$.  
\begin{equation*}
\dtv(x,y) = \max \set{\sum_{i \in I} |x_i - y_i | : \, I \subseteq \set{1, 2, \cdots , n}}
\end{equation*}
Two timbres are close in this metric if they have approximately the same power in every subset of harmonics; they are far apart if there is a large power discrepancy in some subset of harmonics.  From a signal processing point of view, if two timbres are close in this metric, than applying any linear time-invariant filter to the signals will result in outputs that are approximately the same.  
It is well known that total variational distance is related to the $\ell^1$ norm by
\begin{equation*}
\dtv(x,y) = \frac{1}{2} \| x-y \|_1 
\end{equation*}
So with this choice of metric, our sound design problem becomes a constrained $\ell_1$ minimization problem,
\begin{align} \label{e:L1Min}
&\text{Minimize:   $\| x - p \|_1$}  \nonumber \\
&\text{Subject to:  $Hx \le Hb$ component-wise},
\end{align}
where $H$ is the matrix defined in \eqref{e:brightH}.  This problem can be solved efficiently by re-casting it as a linear programming problem; in general solutions are not unique.  Figure \ref{fig:TOSol} shows a solution calculated with the \textsc{Matlab}$^{\textregistered}$ \texttt{linprog} function, with $b$ equal to the timbral vector for the trumpet, and $p$ equal to the timbral vector for the oboe.  Note that the power spectrum of the solution follows that of the trumpet in the higher harmonics, in order satisfy the brightness constraint, while it follows the spectrum of the oboe in the prominent lower harmonics, as it must in order to minimize the distance to the oboe.

\begin{figure} 
\begin{center}
\includegraphics[scale=0.5]{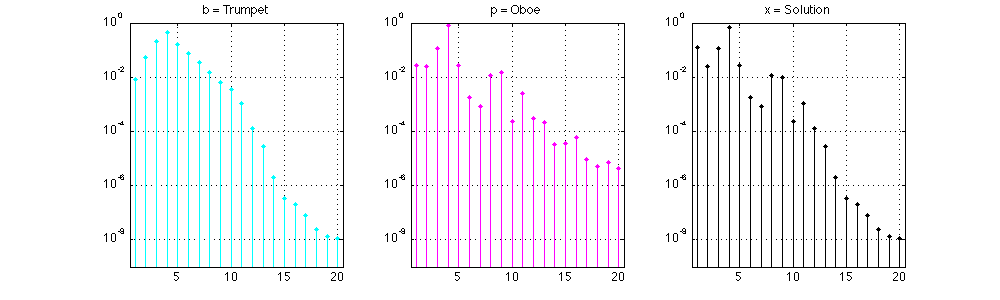} 
\end{center}
\caption{A solution to problem \eqref{e:L1Min}, with $b =$ trumpet and $p =$ oboe.  Solution calculated with the \textsc{Matlab}$^{\textregistered}$ \texttt{linprog} function, using the interior-point algorithm.}
\label{fig:TOSol}
\end{figure}

Of course one would like to refine problem \eqref{e:L1Min} to achieve a unique solution, preferably one that is easily expressible in terms of $b$ and $p$, the data for the problem.  This does not appear to be easy to do.  Using properties of total variational distance, one can show that any solution $x$ of \eqref{e:L1Min}  must satisfy $x \opo p$.  It is well-known that the set of probability vectors ordered by stochastic dominance is in fact a lattice \cite[Theorem 3.1]{Muller}, where the infimum of probability vectors $x$ and $y$ is the probability vector $z = x \meet y$ determined uniquely by $Hz = \min(Hx, Hy)$ component-wise, and $H$ is the matrix defined by \eqref{e:brightH}.
Hence we may conclude that any solution $x$ must satisfy $x \opo b \meet p$, where $b\meet p$ is the infimum of $b$ and $p$.  For probability vectors of length $n = 3$ one can prove that in fact $x = b \meet p$ is the unique solution of \eqref{e:L1Min} that is closest to $b$ in total variational distance.  But for $n \ge 4$ this is no longer true; in fact, in these cases $b \meet p$ need not be a solution of \eqref{e:L1Min} at all!
Another variant of problem \eqref{e:L1Min} that in some sense simultaneously minimizes the distance from $x$ to $p$ and the distance from $x$ to $b$ is
\begin{align} \label{e:L1Min2}
&\text{Minimize:   $\| x - p \|_1 + \| x - b\|_1$} \nonumber \\
&\text{Subject to:  $Hx \le Hb$ component-wise}.
\end{align}
In this case one can prove that $x = b \meet p$ is always a solution, but again, other solutions exist as well when $n \ge 4$.  Whether any of these myriad solutions have significant \emph{perceptible} differences in timbre is an interesting question.

\end{document}